\documentclass[12pt]{article}
\usepackage{amssymb}
\usepackage{amsmath}
\usepackage{amsthm}
\DeclareSymbolFont{extraup}{U}{zavm}{m}{n}
\DeclareMathSymbol{\clubsuit}{\mathalpha}{extraup}{88}

\newcommand\F{{\cal F}}

\newtheorem{thm}{Theorem}[section] \newtheorem{lem}[thm]{Lemma}

\newtheorem*{conj}{Conjecture}
\newtheorem*{note}{Note}

\title{Intersecting  families with bounded intersections}
\author{{\bf Kristina Ago}\\ Department of Mathematics and Informatics\\ University of Novi Sad, Serbia\\ kristina.ago@dmi.uns.ac.rs
\and	{\bf Gyula O.H. Katona}\\
	R\'enyi Institute, HUN-REN\\
	Budapest Pf 127, 1364 Hungary \\ and Mathematical Institute\\ E\"otv\"os Lor\'and University (ELTE),  Hungary\\ ohkatona@renyi.hu}

\begin{document}
	\date{}
	
	\maketitle
	
\begin{abstract}
	Let $\mathcal F\subset 2^{[n]}$ be an $s$-uniform family such that every two distinct sets have a nonempty intersection but intersect in at most $k$ elements. By the well-known Ray-Chaudhuri--Wilson theorem, since the intersections can take at most $k$ different values, we have $|\mathcal F|\leq \binom{n}{k}$. We give a stronger upper bound under our assumptions above, when $n$ is large enough compared to $s$ (and $k+1<s$):  $|\mathcal F|\leq \frac{\binom{n-1}{k}}{\binom{s-1}{k}}$.  This is a special case of an old theorem of Deza, Erd\H os and Frankl, but our proof is simpler and gives a better threshold for $n$. Furthermore we prove  a generalization of the Erd\H os--Ko--Rado theorem for non-uniform families. 	Let $\mathcal F\subset \binom{[n]}{k}\cup\binom{[n]}{k+1}\cup\dots\cup\binom{[n]}{s}$, $3\leq k\leq s$, be a family such that for every two distinct sets
	the size of the intersection is between 1 and $k-1$ and $n$ is large enough then $|\mathcal F|\leq {n-1 \choose k-1}$.
	
	
\end{abstract}

\section{Introduction}

For a positive integer $n$, we write $[n]$ for the set $\{1,2,...,n\}$ and $2^{[n]}$ for the power set of $[n]$. A family $\mathcal F\subset 2^{[n]}$ is called $s$-\emph{uniform} if $|F|=s$ for all $F\in\mathcal F$.  $\binom{[n]}{s}$ denotes the family of all $s$-element subsets of $[n]$, similarly,  $\binom{[n]}{\leq s}$ (respectively $\binom{[n]}{\geq s}$) is the family that contains all at most (respectively at least) $s$-element subsets of $[n]$. We say that the family $\mathcal F$ is \emph{intersecting} if $A\cap B\neq\emptyset$ for every $A,B\in\mathcal F$.

The first result on intersecting families has been the celebrated theorem of Erd\H os, Ko and Rado:

\begin{thm}  {\rm \cite{EKR}}
	If $\mathcal F\subset \binom{[n]}{s}$ is an intersecting family, $n\geq 2s$, then
	\[
	|\mathcal F|\leq\binom{n-1}{s-1}.
	\]
\end{thm}

Since then many intersection theorems have appeared in the literature. Useful overviews are given in the book of Frankl and Tokushige \cite{FT} and in  Chapter 2 of the  book by Gerbner and Patk\'os \cite{GP}. Let us begin by presenting some relevant results for our purposes. For a uniform set system $\mathcal F$, if the possible intersection sizes are limited to at most $k$ distinct values, the Ray-Chaudhuri--Wilson theorem  provides an upper bound on $|\mathcal F|$:

\begin{thm}  {\rm \cite{R-C-W}}
	\label{teo1}
	If $\mathcal F\subset {[n]\choose s}$ is a uniform family such that for every two distinct sets $A,B\in\mathcal{F}$ we have $|A\cap B|\in\{l_1,l_2,\dots,l_k\}$, then
	\[
	|\mathcal F|\leq \binom{n}{k}.
	\]
\end{thm}

If the considered family is not uniform, then the Frankl--Wilson theorem \cite{F-W} provides an upper bound on the cardinality of the system (again, assuming that the intersection sizes are limited to at most $k$ distinct values).

\begin{thm} {\rm \cite{F-W}}
	\label{teo2}
	If $\mathcal F\subset 2^{[n]}$ is a family such that for every two distinct sets $A,B\in\mathcal{F}$ we have $|A\cap B|\in\{l_1,l_2,\dots,l_k\}$, then
	\[
	|\mathcal F|\leq \sum_{i=0}^{k}\binom{n}{i}.
	\]
\end{thm}

The original proofs of Theorems \ref{teo1} and \ref{teo2} rely on techniques from linear algebra. Alon, Babai and Suzuki \cite{ABS} provided a remarkably elegant and concise proof for both theorems using the so-called \emph{polynomial method} \cite{ABS}. They also gave a generalization of the Ray-Chaudhuri--Wilson theorem.

Frankl and  F\"uredi \cite{FF} posed a conjecture that improves Theorem 1.3 in the special case when $l_i=i$ that is when every pair of two distinct sets have a nonempty intersection but intersect in at most $k$ elements. This was proved by Snevily first when $n$ is sufficiently large \cite{snevily21} and later for the general case using the polynomial method.

\begin{thm} {\rm \cite{snevily22}}
	\label{teo3}
	If $\mathcal F\subset 2^{[n]}$ is a family such that for every two distinct sets $A,B\in\mathcal{F}$ we have $1\leq |A\cap B|\leq k$, then
	\[
	|\mathcal F|\leq \sum_{i=0}^{k}\binom{n-1}{i}.
	\]
\end{thm}

The following theorem of Deza, Erd\H os and Frankl gives an improvement of the Ray-Chaudhury theorem for the general case.

\begin{thm}\label{DEF}   {\rm \cite{DEF}} Let $0<k\leq s\leq n$ and $L=\{ \ell_1, \ell_2, \ldots , \ell_k\}$ where $0\leq \ell_1< \ell_2< \ldots < \ell_k$. Suppose that $\F \subset {[n]\choose s}$ is a family such that 
 $|A\cap B|\in\{l_1,l_2,\dots,l_k\}$ holds for every two distinct members $A,B\in\mathcal{F}$. Then for $n>n_0(s,L)$	
$$|\F|\leq \prod_{i=1}^k{n-\ell_i \over s-\ell_i}.$$	
	\end{thm}

A recent improvement for the case $0<\ell_1$ was given by Heged\"us.

\begin{thm} {\rm \cite{H}} Let $0<k\leq s\leq n$ and $L=\{ \ell_1, \ell_2, \ldots , \ell_k\}$ where $0< \ell_1< \ell_2< \ldots < \ell_k$. Suppose that $\F \subset {[n]\choose s}$ is a family such that 
$|A\cap B|\in\{l_1,l_2,\dots,l_k\}$ holds for every two distinct members $A,B\in\mathcal{F}$. Then for $n>n_0(s,L)$	
$$|\F|\leq {n-\ell_1 \choose s}.$$	
\end{thm}

Substituting $\ell_i=i\ (1\leq i\leq k)$ into Theorem \ref{DEF} we obtain
$$	|\mathcal F|\leq \frac{\binom{n-1}{k}}{\binom{s-1}{k}}.$$

The first  goal of our paper is to give an easier proof for this special case. Since we do not use the method of delta-systems our threshold for $n$ is polynomial instead of the general exponential one.
\begin{thm}
	\label{jedini}
	Let $\mathcal F\subset\binom{[n]}{s}$ be a family, $2\leq k,
	k+2\leq s<n$, such that for every two distinct sets $A,B\in\mathcal{F}$ the following holds: $1\leq|A\cap B|\leq k$. Then, if $n\geq s^{5\over 2}$, we have
	\begin{equation}
		\label{glavni}		|\mathcal F|\leq \frac{\binom{n-1}{k}}{\binom{s-1}{k}}.
	\end{equation}
\end{thm}

 It turned out that the method of the proof allows us to prove a generalization of the Erd\H os--Ko--Rado theorem for large $n$.
 
 \begin{thm}
 	\label{lema1}
 	Let $\mathcal F\subset \binom{[n]}{k}\cup\binom{[n]}{k+1}\cup\dots\cup\binom{[n]}{s}$, $3\leq k\leq s$, be a family such that for every two distinct sets $A,B\in\mathcal{F}$ it holds that $1\leq|A\cap B|\leq k-1$. Then, if $n>n_0(s)$, we have
 	\begin{equation}
 		\label{propi}
 		|\mathcal F|\leq\binom{n-1}{k-1}.
 	\end{equation}
 \end{thm}
  
 Section 2 contains the  proofs of our theorems.  In Section 3 we consider some related constructions.
 At the end of the article, we present some open problems which may serve as a basis for further investigations.

\section{Proofs}

{\it Proof of Theorem 1.7}.  First, if there is an element, for example $1$, such that $1\in F$ for all $F\in\mathcal F$, then for all those $k$-element subsets that do not contain $1$ it holds that they are included in at most one element of $\mathcal{F}$. Since each $F\in\mathcal F$ contains $\binom{s-1}{k}$ such $k$-element sets, we have $|\mathcal F|\binom{s-1}{k}\leq\binom{n-1}{k}$, thus (\ref{glavni}) follows even without any restriction on $n$.
	
	Next, we consider the case when there is no element which is contained in all the sets, but there are two elements, for example $1$ and $2$, such that for all $F \in \mathcal{F}$, either $1 \in F$ or $2 \in F$ (or both). Since neither $1$ nor $2$ is contained in all sets, it follows that there exist sets $A$ and $B$ such that $1\in A$ but $2\not\in A$ and $2\in B$ but $1\not\in B$. For $b\in B$ (respectively $a\in A$) let $\mathcal F(1,b)$ (respectively $\mathcal F(2,a)$) denote those sets $F$ in $\mathcal F$ for which it holds that $\{1,b\}\subset F$ (respectively $\{2,a\}\subset F$). Now for all those $(k-1)$-element subsets that do not contain $1$ and $b$ it holds that they are included in at most one member of $\mathcal{F}(1,b)$. Since each $F\in\mathcal F(1,b)$ contains $\binom{s-2}{k-1}$ such $(k-1)$-element sets, we have $|\mathcal F(1,b)|\binom{s-2}{k-1}\leq\binom{n-2}{k-1}$, thus $|\mathcal F(1,b)|\leq\frac{\binom{n-2}{k-1}}{\binom{s-2}{k-1}}$ and the same holds for $|\mathcal F(2,a)|$.
	
	Let $F\in \mathcal{F}$ be a member different from $A$ and $B$. If $1\in F$ then $F\in \mathcal F(1,b)$ must hold for some element $b\in B$ since $\mathcal F$ is intersecting. Similarly, $2\in F$ implies $F\in \mathcal F(2,a)$ for an $a\in A$. But $A$ and $B$ are also intersecting, say in an element $c$. Then $A\in  \mathcal F(1,c)$ and $B\in \mathcal F(2,c)$ must also hold.
	
	All together we have
	$$\mathcal F=\bigcup_{b\in B}\mathcal F(1,b)\cup \bigcup_{a\in A}
	\mathcal F(2,a)$$
	and
	\[
	|\mathcal F|\leq\sum_{b\in B}|\mathcal F(1,b)|+\sum_{a\in A}|\mathcal F(2,a)|\leq2s\frac{\binom{n-2}{k-1}}{\binom{s-2}{k-1}}.
	\]
	Thus, we have to prove
	\[
	2s\frac{\binom{n-2}{k-1}}{\binom{s-2}{k-1}}\leq\frac{\binom{n-1}{k}}{\binom{s-1}{k}},
	\]
	which is equivalent to
	\begin{equation}
		\label{krajdrugog}
		2s\leq\frac{n-1}{s-1}.
	\end{equation}
	Since $2s\leq\frac{n}{s}$ is a stronger condition than (3), it follows that, under the assumption $n\geq 2s^2$, the desired inequality (1) holds. Since $4\leq s$ the condition $s^{5\over 2}\leq n$ implies $2s^2\leq n$, finishing the proof in this case.
	
	Finally, if there are no two elements such that at least one of them is contained in every member of $\mathcal{F}$, then we will use the following lemma.
	
	\begin{lem}
		\label{lr}
		Suppose that $3\leq s$ is an integer and let $\mathcal F\subset \binom{[n]}{s}$ be an intersecting family,
		such that for any pair $u,v \in [n]$ there is a member $F\in \mathcal F$ with $F\cap \{u,v\}=\emptyset$. Then there is a family $\mathcal T\subset \binom{[n]}{3}$ such that every member of $\mathcal F$ contains at least one member of $\mathcal T$ and $|\mathcal T|\leq s^3$.
	\end{lem}

\begin{proof} Let $A\in \mathcal F$ be a fixed ``starting member" of $\mathcal F$ and choose an arbitrary element $a\in A$.  As $a$ is not contained in all sets in $\mathcal F$, there is a set $B(a)\in\mathcal F$ such that $a\not\in B(a)$. Let $b\in B(a)$ any element of $B(a)$, $b\neq a$. Since, by the assumptions, there are no $2$ points such that at least one of them is contained in each member of $\mathcal{F}$, there is a set $C(a,b)$ such that $a,b\not\in C(a,b)$. Let $c\in C(a,b)$, $c\neq a,b$. There are at most $s^3$ such $3$-element sets $\{a,b,c\}$.
$\mathcal T$ will be the family of such triples $\{a,b,c\}$.

We only have to prove that there exists a $T\in \mathcal T$ for every
$F\in \mathcal F$ satisfying $T\subset F$. Let us consider a member
$F\in \mathcal F$. Since $ \mathcal F$ is intersecting, $A\cap F$ is non-empty, therefore contains an element $a$.  Neither $B(a)\cap F$ is empty, it contains an element $b$. Finally $C(a,b)\cap F$ also has an element $c$, proving that $T=\{ a,b,c\}\subset F$ holds, as desired.\end{proof}

In order to finish the proof  we suppose that there are no $2$ elements such that at least one of them is contained in each member of $\mathcal{F}$. Introduce the notation
$\mathcal F(U)=\{ F\in \mathcal F: U\subset F\}$ for a set $U\subset [n]$.
 By the lemma we have
 $$\mathcal F\subset \bigcup_{T\in \mathcal T}\mathcal F(T).$$

 Now introduce the notation $\widehat{\mathcal F}(T)=\{F-T:\ F\in \mathcal F(T)\}$. This is a family of $s-3$-element subsets of $[n]-T$ satisfying the property that their $k-2$-element subsets are all different.
  Since each $F-T\in \widehat{\mathcal F}(T)$ contains $\binom{s-3}{k-2}$ such $(k-2)$-element sets, we have $|\widehat{\mathcal F}(T)|\binom{s-3}{k-2}\leq\binom{n-3}{k-2}$, thus $|\mathcal F(T)|=|\widehat{\mathcal F}(T)|\leq\frac{\binom{n-3}{k-2}}{\binom{s-3}{k-2}}$. All together we have
\[
|\mathcal F|\leq |\mathcal T|\frac{\binom{n-3}{k-2}}{\binom{s-3}{k-2}}\leq s^3\frac{\binom{n-3}{k-2}}{\binom{s-3}{k-2}}.\eqno(4)
\]
by the lemma.
Thus, we have to prove
\[
s^3\frac{\binom{n-3}{k-2}}{\binom{s-3}{k-2}}\leq\frac{\binom{n-1}{k}}{\binom{s-1}{k}},
\]
which is equivalent to

$$	s^3\leq\frac{(n-1)(n-2)}{(s-1)(s-2)}.\eqno(5)$$

Since $s^3\leq\frac{n^2}{s^2}$ is a stronger condition than  (5), it follows that, under the assumption $n\geq s^{\frac{5}{2}}$, the desired inequality (1) holds.\qed

\vskip 3mm

{\it Proof of Theorem 1.8}.
 	If $k=s$ then the theorem gives back the Erd\H os--Ko--Rado theorem, therefore we can suppose throughout the proof that $k<s$. First, if there is an element, for example $1$, such that $1\in F$ for all $F\in\mathcal F$, then for all those $(k-1)$-element subsets that do not contain $1$ it holds that they are included in at most one member of $\mathcal{F}$. Therefore $|\mathcal F|$ cannot exceed the total number of such $k-1$-element sets, that is
 	\[
 	|\mathcal F|\leq\binom{n-1}{k-1},
 	\]
 	thus (\ref{propi}) follows even without any restriction on $n$.
 	
 	Next, if there is no element which is contained in all the sets, but there are two elements, for example $1$ and $2$, such that for all $F \in \mathcal{F}$, either $1 \in F$ or $2 \in F$ (or both), then there exist sets $A$ and $B$ such that $1\in A$ but $2\not\in A$ and $2\in B$ but $1\not\in B$. For $b\in B$ let $\mathcal F(1,b)$ denote those sets $F$ in $\mathcal F$  for which it holds that $\{1,b\}\in F$, and for $a\in A$ let $\mathcal F(2,a)$  be defined in the same way. Similarly to the previous case, a $k-2$-element subset of $[n]-\{ 1, b\}$ cannot be included in two members of $\mathcal F(1,b)$ that is the total number of members of $\mathcal F(1,b)$ cannot exceed the number of such $k-2$-element sets:
 	\[
 	|\mathcal F(1,b)|\leq\binom{n-2}{k-2},
 	\]
 	and the same holds for $|\mathcal F(2,a)|$.
 	Now,
 	\[
 	|\mathcal F|\leq\sum_{b\in B}|\mathcal F(1,b)|+\sum_{a\in A}|\mathcal F(2,a)|\leq2s\binom{n-2}{k-2}.
 	\]
 	Thus, we have to prove
 	\[
 	2s\binom{n-2}{k-2}\leq\binom{n-1}{k-1}.
 	\]
 	This is equivalent to
 	\[
 	2s\leq\frac{n-1}{k-1},
 	\]
 	which is true if $n$ is sufficiently large compared to $s$.
 	
 	Finally, if there are no two elements such that at least one of them is contained in each member of $\mathcal F$, we can follow the corresponding part of the proof of Theorem \ref{jedini}. First observe that Lemma 2.1 remains valid if $\F \subset {[n]\choose s}$ is replaced by $\F \subset {[n]\choose k}\cup \ldots \cup{[n]\choose s}$ where $3\leq k$, since we used in the proof  only the condition that the sizes of the members of $\F$ are between 3 and $s$.
 	
 	Then we can continue in a similar manner as in the proof of Theorem \ref{jedini}. Instead of the equation (4), the corresponding argument now leads to $|\mathcal F|\leq s^3\binom{n-3}{k-3}$. However we need to show that it is not more than ${n-1\choose k-1}$:
 	$$s^3\binom{n-3}{k-3}\leq {n-1\choose k-1}.$$
 	But this obviously holds for sufficiently large $n$.
 \qed

 \begin{note} Theorem \ref{lema1} implies Snevily's theorem for $3\leq k$ if the members $F$ of the family satisfy $|F|\leq s$ and $n$ is sufficiently large relative to $s$.
 	\end{note}

 \begin{proof} Let $\mathcal F=\mathcal F_1\cup \mathcal F_2\cup\dots\cup\mathcal F_{s}$, where $\mathcal F_i$ denotes the collection of the $i$-element sets in $\mathcal F$. Using the Erd\H os--Ko--Rado theorem separately for the families $\mathcal F_i$, $i=1,2,\dots,k-1$, we have $|\mathcal F_i|\leq\binom{n-1}{i-1}$ (we can use it since in these systems the sizes of the intersections are certainly not greater than $k-1$, so we only have to be sure that they are nonempty). For the rest of the system, if $n$ is large enough, we use Theorem \ref{lema1}, and thus together we have the desired result.
 \end{proof}

\section{Related constructions}

 Our proof for Theorem 1.7 does not work when $k=1$, since it contains $k-2$-element sets. But of course it is true for sufficiently large $n$. 
The lines of a finite projective plane geometry give a counter-example for small $n$. Let $q$ be a prime power. It is well-known that there are $q^2+q+1$ lines (subsets)
in the $q^2+q+1$-element set, each of size $q+1$ and their pairwise intersections have exactly one element. Taking $n=q^2+q+1,  s=q+1$ and $k=1$ we have $|\F|=n$ which is definitely more than the bound in the theorem. Therefore $n$ must be at least $s^2$.

Let us now show constructions attaining the upper bound in Theorem 1.7. Let ${\cal S}$ be an $(n-1, s-1, k)$-Steiner system, that is ${\cal S}\subset {[n-1]\choose s-1}$ where every $k$-element subset of $[n-1]$ is a subset of exactly one member of ${\cal S}$. If we have such a system, add a new element, say $n$ to all members. The so obtained family of $s$-element subsets of $[n]$ give equality in Theorem 1.7.

Now some counter-examples will be shown to the statement of Theorem 1.7 when $n$ is relatively small. For instance, consider the complements of the lines in the Fano plane. In this case $n=7$, $s=4$ and $k=3$, so Theorem \ref{jedini} would imply $|\mathcal F|\leq 5$, whereas we know that $|\mathcal F|=7$. This counterexample belongs to a (somewhat) larger class of counterexamples determined by certain designs. (We apologize the reader that in the rest of this section the terminology of the design theory will be used rather than the usual one in extremal set theory.) Let $X$ be a finite set of $n$ points and $\beta$ be a finite family of distinct $s$-subsets of $X$, called \emph{blocks.} Then the pair $D=(X,\beta)$ is called a \emph{$k$-$(n,s,\lambda)$ design} if every $k$-subset of $X$ occurs in exactly $\lambda$ blocks.  In a $k$-$(n,s,\lambda)$ design the number of the blocks $b$ is $b=\frac{\lambda\binom{n}{k}}{\binom{s}{k}}$. Thus, for the number of the blocks in a $2$-$(n,s,\lambda)$ design we have $b=\frac{\lambda n(n-1)}{s(s-1)}$. A $2$-$(n,s,\lambda)$ design is called \emph{symmetric} if $b=n$ (this is one of a few equivalent conditions). From these relations we have that in case of symmetric $2$-$(n,s,\lambda)$ designs it holds that
$$n=\frac{s(s-1)}{\lambda}+1.\eqno(6)$$

It is easy to see that in a symmetric $2$-$(n,s,2)$ design every two distinct blocks have exactly 2 points in common. Let us consider a family $\mathcal F$ that is determined by a symmetric $2$-$(n,s,2)$ design (these designs are called \emph{biplanes}). The family $\mathcal F$ is defined on $n$ elements, thus, from (6) we have $n=\frac{s(s-1)}{2}+1$, all sets in this family have $s$ elements, while the size of the intersection of every pair of its sets is $2$. It can be calculated that for these parameters Theorem \ref{jedini} would give
\[
|\mathcal F|\leq\frac{(n-1)(n-2)}{(s-1)(s-2)}=\frac{\frac{s(s-1)}{2}(\frac{s(s-1)}{2}-1)}{(s-1)(s-2)}\leq \frac{s(s+1)}{4},
\]
whereas, by the listed properties of symmetric designs, we have
\[
|\mathcal F|=\frac{s(s-1)}{2}+1.
\]

Thus, every symmetric $2$-$(n,s,2)$ design is indeed a counterexample to our theorem for small $n$. Such block designs exist for the following values of $s$: $4,5,6,9,11,13$. 

The family defined by certain \emph{quasi-residual} designs also serve as a counter-example where not all the intersections have the same size. For a symmetric $D=(X,\beta)$ design its \emph{residual} is a design $D'=(X,\beta')$  in which $\beta'=\{B- B':B\in\beta, B'\neq B {\rm \ for\ some\ } B'\in\beta\}$.
A {\it residual design} of a symmetric $2$-$(n,s,2)$ design has the following properties. The number of the points is $N=n-s$, the sizes of its blocks are $S=s-2$, while the number of its blocks is $B=\frac{s(s-1)}{2}$. A design is called {\it quasi-residual} if its parameters correspond to the parameters of a residual design of a symmetric design. The following theorem will be used.
\begin{thm} {\rm \cite{lawless}}
	In a quasi-residual design with $\lambda=2$ we have $N=\frac{S(S+1)}{2}$ points, the size of the blocks is $S$, the number of the blocks is $B=\frac{(S+2)(S+1)}{2}$, while the size of the intersection of every pair of its sets is $1$ or $2$.
\end{thm}

Let us consider a family $\mathcal F$ which is determined by a quasi-residual $2$-$(N,S,2)$ design. It can be calculated that for these parameters Theorem \ref{jedini} would give
\[
|\mathcal F|\leq\frac{(n-1)(n-2)}{(s-1)(s-2)}=\frac{(\frac{S(S+1)}{2}-1)(\frac{S(S+1)}{2}-2)}{(S-1)(S-2)}= \frac{(S+2)(S^2+S-4)}{4(S-2)},
\]
whereas, by properties of these designs, we have
\[
|\mathcal F|=\frac{(S+2)(S+1)}{2}.
\]

In the counterexamples discussed so far, the families were intersecting but the intersection of any two sets always contained at most two elements. However, counterexamples with three-element intersections also exist. Such cases can be found among \emph{quasi-symmetric} designs. A design is quasi-symmetric if there exist two integers $\mu_1$ and $\mu_2$ such that any two blocks intersect in either $\mu_1$ or $\mu_2$ points. It is known (see \cite{handbook}) that there exist quasi-symmetric designs with the following parameters (for all of them $(\mu_1,\mu_2)=(1,3)$): $2$-$(21,7,12)$, $2$-$(22,7,16)$, $2$-$(23,7,21)$, $3$-$(22,7,4)$ and $4$-$(22,7,4)$. In each case, the reader can compute the corresponding values and verify that the cardinalities of the families they define exceeds the limit allowed by Theorem \ref{jedini}.

Finally let us try to investigate, how large $n$ must be in Theorem \ref{lema1}. Since even the Erd\H os--Ko--Rado theorem contains the condition $2k\leq n$ we cannot expect a better bound here. Indeed, all the $k$-element subsets of $[2k-1]$ satisfy the condition of Theorem \ref{lema1} and this is more than ${n-1\choose k-1}$.  The following construction gives slightly larger $n$ for which the statement of the theorem does not hold.

Let $d\geq 2$ be an integer such that $d-1|k-2$  and choose $n=2k-2+{k-2\over d-1}$. It is easy to see that $n-k=k-2+{k-2\over d-1}$ is divisible by $d$. Hence the set $\{ k+1, \ldots , n\} =\{ k+1, \ldots , 2k-2+{k-2\over d-1}\}$ can be partitioned into $d$ equally sized subsets: $A_1, \ldots , A_d$. Let $B_i=\{ 2, 3, \ldots , k\} \cup A_i$. Our constructed  family consists of all $k$-element subsets containing the element 1 and $B_1, B_2 , \ldots , B_d$. It has ${n-1\choose k-1}$ members of size $k$ and $d$ members of size $s=k-1+{k-2\over d-1}$. It is easy to see that this family is intersecting but the sizes of the intersections cannot exceed $k-1$. However the number of members is
$${n-1\choose k-1}+d.$$

$n$ will be the largest in this construction if $d$ is chosen to be 2. Then $n=3k-4$, more than the trivial bound $2k$. But its ratio is worst if it is compared to $s$. 

 \section{Open problems}
 
 {\bf Open problem.} We believe that the condition $n\geq s^{5\over 2}$ in Theorem 1.7.
 	can be replaced by $n\geq s^{2+{1\over k-1}}$. What is the real threshold?
\vskip 3mm
 
 \begin{conj} Suppose that $0 \leq \ell \leq k\leq s$
 	are integers, and let $\mathcal F\subset \binom{[n]}{\ell }\cup \ldots \cup \binom{[n]}{k}\cup\dots\cup\binom{[n]}{s}$, is a family such that for every two distinct sets $A,B\in\mathcal{F}$ it holds that $1\leq|A\cap B|\leq k-1$. Then, if $n>n_0(s)$, we have
 	\begin{equation*}
 		\label{propi}
 		|\mathcal F|\leq \sum_{i=\ell}^k\binom{n-1}{i-1}.
 	\end{equation*}
 \end{conj}
 
 Bal\'azs Patk\'os \cite {P} and independently Chunyang Dou and Xing Peng \cite{DP} noticed that this conjecture can easily be proved by partitioning the family $\F$ according to the sizes of the members. Erd\H os-Ko-Rado theorem can be applied for the sets of size $i\ (\ell\leq i<k)$ and Theorem 1.8 can be applied for the rest of the family.
 
\section*{Acknowledgements}
We are indebted to professor G\'abor Heged\"us for his very helpful remarks.

The first author was supported by the Ministry of Science, Technological Development and Innovation of the Republic of Serbia (grants no. 451-03-33/2026-03/200125 and 451-03-34/2026-03/200125) and by the Hungarian Academy of Sciences (the \emph{Domus} program).

The work of the second author was not
	supported by the Hungarian National Research, Development and Innovation Office.
	
\section*{Declarations}

{\bf Competing interests:} The authors have no relevant financial or non-financial interests to disclose.\\

\noindent{\bf Data availability:} Not applicable.

\end{document}